\documentclass[a4paper,11pt]{amsart}
\usepackage{amssymb,amsthm,bbm,epic,eepic,graphics,amsbsy}
\usepackage{a4wide}
\usepackage{amssymb,amsthm}
\usepackage{xypic}
\usepackage[applemac]{inputenc}
\newcommand{\la}{\lambda}

\def\Hom{\mathrm{Hom}}

\def\Aut{\mathrm{Aut}}

\def\U{\mathsf{U}}

\def\C{\ensuremath{\mathbbm{C}}}
\def\Q{\mathbbm{Q}}
\def\Qb{\overline{\Q}}
\def\Z{\mathbbm{Z}}

\def\bar{\overline}

\def\k{\mathbbm{k}}
\def\kt{\mathbbm{k}^{\times}}
\def\Ass{\mathbbm{Assoc}}

\def\eps{\epsilon}

\def\GT{\mathrm{GT}}
\def\GTH{\widehat{\mathrm{GT}}}

\def\onto{\twoheadrightarrow}
\def\into{\hookrightarrow}
\def\ii{\mathrm{i}}
\def\Gal{\mathrm{Gal}}

\def\GL{\mathrm{GL}}
\def\PGL{\mathrm{PGL}}
\def\SL{\mathrm{SL}}

\def\ZH{\widehat{\Z}}
\def \limpro#1{\lim\limits_{\displaystyle\longleftarrow\atop {#1}}}
\newtheorem{theo}{Theorem}[section]
\newtheorem{prop}[theo]{Proposition}
\newtheorem{defi}[theo]{Definition}
\newtheorem{lemma}[theo]{Lemma}

\def\mod{\ \mathrm{mod}\ }

\makeatletter
\DeclareRobustCommand\widecheck[1]{{\mathpalette\@widecheck{#1}}}
\def\@widecheck#1#2{%
   \setbox\z@\hbox{\m@th$#1#2$}%
   \setbox\tw@\hbox{\m@th$#1%
      \widehat{%
         \vrule\@width\z@\@height\ht\z@
         \vrule\@height\z@\@width\wd\z@}$}%
   \dp\tw@-\ht\z@
   \@tempdima\ht\z@ \advance\@tempdima2\ht\tw@ \divide\@tempdima\thr@@
   \setbox\tw@\hbox{%
      \raise\@tempdima\hbox{\scalebox{1}[-1]{\lower\@tempdima\box\tw@}}}%
   {\ooalign{\box\tw@ \cr \box\z@}}}
\makeatother

\title[Characters of the Grothendieck-Teichm\"uller group]{{\bf Characters of the Grothendieck-Teichm\"uller group
through rigidity of the Burau representation}}
\author{Ivan Marin}
\date{June 13th, 2007}
\begin{document}

\maketitle

\bigskip

\bigskip

\vspace{-1cm}
\begin{center} Institut de Math\'ematiques de Jussieu\\
Universit\'e Paris 7\\
175 rue du Chevaleret\\
 F-75013 Paris.
\end{center}
\medskip

\noindent {\bf Abstract.} We present examples of characters
of absolute Galois groups of number fields that can be recovered through
their action by automorphisms on the profinite completion of
the braid groups, using a ``rigidity'' approach.
The way we use to recover them is through classical representations
of the braid groups, and in particular through the Burau representation.
This enables one to extend these characters to Grothendieck-Teichm\"uller
groups.
\medskip


\section{Introduction}

\subsection{Braid groups and free groups}

Let $n \geq 3$. We define the braid group $B_n$ as follows. Let $D_n$
be the disk with boundary with $n$ points removed (see figure \ref{figDn}).
Then $B_n$ is the group $\mathrm{Diff}^{+}(D_n)/\mathrm{Diff}^{+}_0(D_n)$
of diffeomorphisms which preserve orientation and fix the boundary pointwise
modulo isotopy. We note that the fundamental group
of $D_n$ is the free group $F_n$ on $n$ generators, hence we get from the
definition a natural action by automorphisms of $B_n$ on $F_n$.
This action is called the Artin action and is known to
be faithful, hence $B_n$ can be considered as a subgroup of $\Aut(F_n)$.

\subsection{Analogy : rigid local systems}

For motivation, let us say that an irreducible representation
$R : F_n \to \GL(V)$ is B-rigid if, for all $g \in B_n$,
we have $R \circ g \simeq R$ as representations of $F_n$. In
that case, we get a well-defined projective representation
$Q_R : B_n \to \PGL(V)$ of the braid group $B_n$. It turns
out that many B-rigid representations of the free group
are known. They are constructed by the way of so-called \emph{rigid local
systems}. We recall what this is, refering to \cite{KATZ} for
further reading.

Call a $(n+1)$-tuple $\underline{A} = (A_1,\dots
,A_{n+1}) \in \GL(V)^{n+1}$ a local system if $A_1 A_2 \dots A_{n+1} = 1$.
Another local system $\underline{B}$ is considered
equivalent to $\underline{A}$ if there exists $P \in \GL(V)$
such that $PA_i = B_i P$ for all $i \in [1,n+1]$.
Geometrically, this is a local system on $\mathbbm{P}^1(\C)$ minus
$n+1$ points, which is homotopically equivalent to $D_n$. We call it
irreducible if $V$ is irreducible under the action of
$A_1,\dots,A_{n+1}$. Note that this is the same as to ask it
to be irreducible under the action of $A_1,\dots,A_{n}$,
as $A_{n+1} = (A_1\dots A_{n})^{-1}$. We call such a local system
\emph{rigid} if, for any irreducible local system $\underline{B}$
such that $A_i$ and $B_i$ are conjugated in $\GL(V)$ for all
$i \in [1,n+1]$, then $\underline{B}$ is equivalent to $\underline{A}$.

An explicit description of the Artin action enables one to prove
the following : if $\underline{A}$ is a rigid (irreducible) local
system such that $A_i$ is conjugated to $A_j$ in $\GL(V)$
for all $i,j$, then the associated irreducible representation
$R : F_n \to \GL(V)$ which sends the $i$-th generator of $F_n$
to $A_i$ is B-rigid, hence defines a projective representation
of $B_n$ over $V$.

\subsection{Rigidity of Braid group representations}

In this paper, we show how one can apply the same approach to the
Grothendieck-Teichmüller group $\GTH$, being defined as a group
of automorphism of some completion of the braid group $B_n$.
In this way we get characters of natural subgroups of $\GTH$.

\medskip

One of the main interest of the group $\GTH$ is that it contains
the absolute Galois group $\Gal(\Qb|\Q)$.
Here we pay special attention to this aspect.
We show in particular how some of the classical characters
of absolute Galois groups of number fields, namely the Kummer character $\rho_2$
and the Soulé characters $\kappa_3$ and $\kappa_5$ can be
extended to the Grothendieck-Teichmüller group by using rigidity
of representations of the braid groups.

\medskip

In \S 2 and \S 3 we recall classical facts about completions of groups
and about the various groups involved here. Section 4 recovers the Kummer
character of roots of 2 by rigidity of a natural representation
$B_3 \to \SL_2(\Z)$, which can be seen as a degeneration of the Burau
representation of $B_3$. Section 5 defines
the notion of $\GT$-rigidity, and section 6 applies this mecanism
to the classical Burau representation.

\medskip

Part of the results exposed here have been proved in \cite{GT}.
This paper is roughly based on a talk given in Toulouse on January 26th,
2007, during the
``S\'eminaire M\'editerran\'een d'Alg\`ebre et de Topologie''
dedicated to the Grothendieck-Teichmüller group,
which was organised by C. Kapoudjian and I. Potemine.

\begin{figure}
\begin{center}
\mbox{\hspace{-3.5cm} \includegraphics{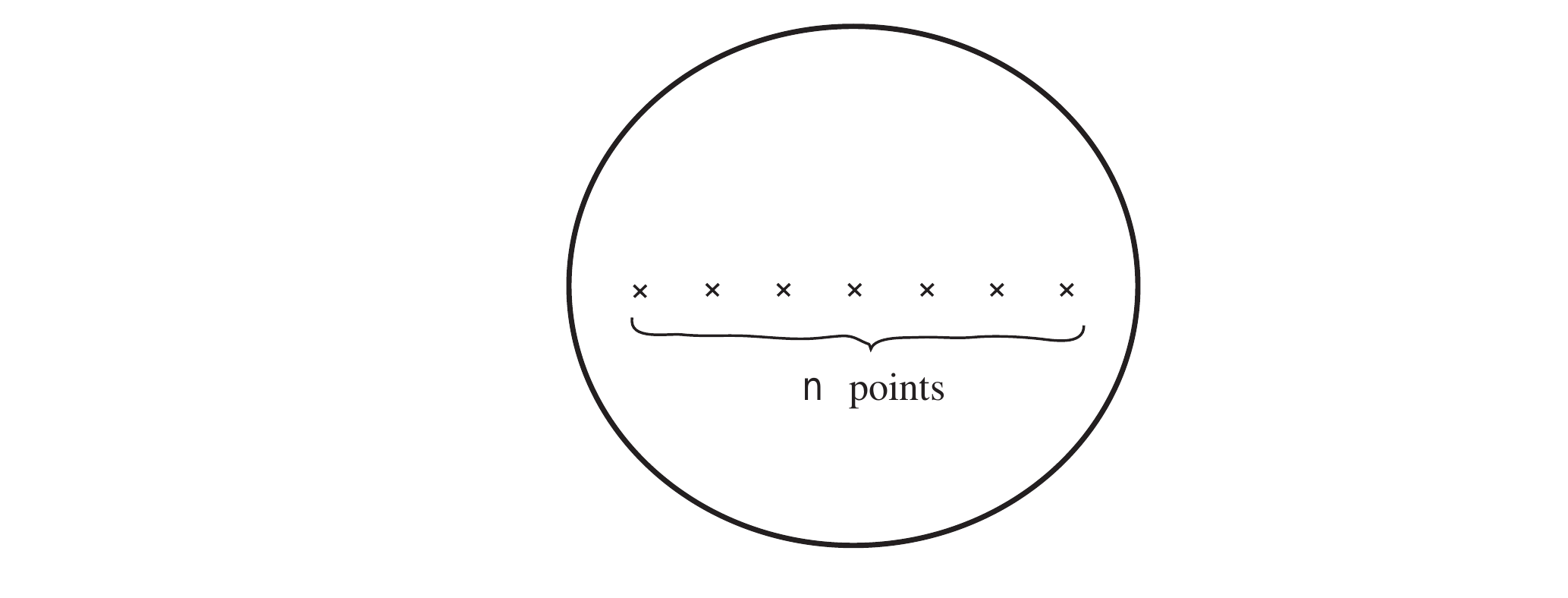}}
\end{center}
\caption{Disk with $n$ points removed}
\label{figDn}
\end{figure}

\section{Completions}

\subsection{General facts}

If $G$ is an arbitrary group and $\mathcal{Q}$ is a class
of groups satisfying certain well-known conditions, the completion $G_{\mathcal{Q}}$
of $G$ with respect to $\mathcal{Q}$ is the limit of the direct
system formed by the quotients of $G$ which belong to $\mathcal{Q}$.
We consider here the following classes : the finite groups,
the nilpotent groups, the finite nilpotent groups and the finite $\ell$-groups
for some fixed prime number $\ell$. We denote the corresponding
completions $\widehat{G}$, $G_{nil}$, $G_{f-nil}$, $G^{\ell}$.

Completions with respect to subclasses of the class of finite groups
are so-called profinite : with respect to the topology inherited from
the discrete finite quotients, they are compact and totally discontinuous.
It follows that the natural (continuous) morphism $\widehat{G} \to G_{f-nil}$
is always surjective. Another general feature is that
finite nilpotent groups are direct products of their $\ell$-Sylows.
A consequence here is that $G_{f-nil}$ is canonically isomorphic
to the direct product of the groups $G^{\ell}$, for $\ell$
a prime number. Finally, nilpotent groups are residually finite.
Recall that a group is said to be residually $\mathcal{Q}$ if every
non-trivial element admits a non-trivial image in
some group in $\mathcal{Q}$. A consequence here is that the natural morphism
$G_{nil} \to G_{f-nil}$ is always injective.

In case $G$ is a free group $F$, it is known
that $F$ is residually nilpotent, and in particular residually finite.
It follows that the morphisms $F \to F_{nil}$ and $F \to \widehat{F}$
are injective. The whole picture is illustrated by the following
commutative diagram.

\begin{center}
\mbox{\xymatrix{
 & \widehat{F} \ar@{->>}[dr] & \\
F \ar@{^{(}->}[ur] \ar@{^{(}->}[dr] & & F_{f-nil} \ar[r]^{\simeq} & {\displaystyle \prod_{\ell} F^{\ell}}  \\
 & F_{nil} \ar@{^{(}->}[ur] & }}
\end{center}

Also note that the universal property of free groups
extends to their completions. Assume for simplicity
that $F$ is free on two generators $x,y$. Choosing
two elements $\bar{x},\bar{y}$ in an arbitrary
group $G$ amounts to defining a morphism $\varphi : F \to G$,
such that $\varphi(f) = f(\bar{x},\bar{y})$.
If $G$ is finite (respectively nilpotent, finite and nilpotent,
or a finite $\ell$-group), this morphism extends to
the profinite (respectively pro-nilpotent, pro-finite-nilpotent,
pro-$\ell$) completion of $F$. In the general case it extends to
a morphism between the corresponding completions of $F$ and $G$. If $f$ lies in one these completions,
we will denote $f(x,y)$ the image of $x,y$ under this
morphism.

The braid groups are residually finite, so $B_n$ embeds into
$\widehat{B_n}$. This can be deduced from the same statement
for the free groups, as automorphism groups of residually finite
groups are residually finite by a classical result of Baumslag.

However, $B_n$ does not embed in $(B_n)_{nil}$ for $n \geq 3$. This comes from the fact
that its abelianization $B_n / C^1 B_n$ is isomorphic to $\Z$.
Indeed, if $G / C^1 G \simeq \Z$ it is known (see \cite{BOURB} \S 4 ex. 7)
that $C^{r+1} G = C^r G$
for all $r \geq 1$, whence $G_{nil} = G/C^1 G = \Z$
(and $G^{\ell} \simeq \Z_{\ell}$).

\subsection{Maltsev Hull}

Let $N$ be a torsion-free nilpotent group, assumed to
be finitely generated for simplicity. Recall that a group is called
divisible if, for all $r \in \Z \setminus \{ 0 \}$, the map
$x \mapsto x^r$ is a bijection. The Maltsev hull of $N$ is defined by the
following classical result.

\begin{theo} (Maltsev) There exists a unique embedding $j : N \into
N \otimes \Q$ where $N \otimes \Q$ is a divisible nilpotent
group such that $\forall x \in N \otimes \Q \ \ \exists r \in
\Z \setminus \{ 0 \} \ \ x^r \in j(N)$.
\end{theo}
 
\medskip

It is also known that these divisible nilpotent groups are
unipotent algebraic groups over $\Q$ (see the \cite{BOURB} for the
full details and \cite{DELIGNE} \S 9 for a crash course on this
correspondence).
It follows that we can define $N \otimes \k$, for an arbitrary
communitative $\Q$-algebra $\k$ with unit, by the $\k$
points of the unipotent algebraic group $N \otimes \Q$.
Schematically, $N \otimes \k = (N \otimes \Q)(\k)$.

\medskip

\noindent {\bf  Example 1.} (see \cite{BOURB} \S 4 ex. 15 and ex. 16.) Assume that $N < T_n(\Z)$ where
$T_n(A)$ denote the set of unipotent upper-triangular matrices with coefficients
in $A$. Then $N < T_n(\Q)$ and $N \otimes \Q$ can be chosen to be the
subgroup of $T_n(\Q)$ defined by $\{ g \in T_n(\Q) \ | \ \exists r \in \Z\setminus \{ 0 \} \  \ g^r
\in N \}$. This example is universal in the sense that
any finitely generated torsion-free unipotent group can be
embedded in some $T_n(\Z)$. Moreover $N \otimes \k$ can be seen inside $T_n(\k)$
in the same way.

\medskip

\noindent {\bf  Example 2.} Let $H$ be the torsion-free nilpotent group defined
by the presentation $<A,B,C | (A,B)=C, (A,C)=(B,C)=1 >$. It is
known and easily checked that $H$ admits a normal form of the following
type : every $h \in H$ can be expressed uniquely as $A^{\alpha}
B^{\beta} C^{\gamma}$ for some $(\alpha,\beta, \gamma) \in \Z^3$.
In other words, the group $(H,.)$ can be identified to some
group structure $(\Z^3, \star)$ on $\Z^3$. This group law on $\Z^3$ is clearly
polynomial, more precisely given by
$$
(u,v,w) \star (x,y,z) = (u+x,v+y,w-xv + z)
$$
and thus it extends to a group structure on $\Q^3$. Again, the same
construction yields $H \otimes \k$ by replacing $\Q^3$ by $\k^3$.

\subsection{Maltsev completion of the free groups}

A free group $F$ is not torsion-free nilpotent, however it has the nice property that
each nilpotent quotient $F/C^r F$ is torsion-free. It follows that $F$ embeds into
$$
F(\k) = \limpro{r} F/C^r F \otimes \k
$$
which is called the Maltsev or the pro-$\k$-unipotent completion
of $F$. A non-canonical but explicit description of $F(\k)$
can be obtained as follows. Assume for convenience that $F$ is a free
group on two generators $x,y$ and consider two non-commutative variables
$A$ and $B$. We form
\begin{itemize}
\item the free associative $\k$-algebra with unit on $A,B$, denoted
$\k<A,B>$ (``non-commutative polynomials in $A,B$ '')
\item the completion $\k\ll A,B \gg$ of the former, with respect to
the grading $\deg A = \deg B = 1$ (``non-commutative formal series in $A,B$ '')
\item the free Lie $\k$-algebra on $A,B$, denoted
$\mathfrak{F}$ (``Lie polynomials in $A,B$ '')
\end{itemize}
The Lie algebra $\mathfrak{F}$ is also graded by $\deg A = \deg B = 1$
and it is a classical fact that the completion $\widehat{\U \mathfrak{F}}$
of its universal envelopping algebra is canonically isomorphic to
$\k \ll A , B \gg$. In particular, the group $\exp \widehat{\mathfrak{F}}$
can be seen as a group of invertible elements in $\k \ll A,B \gg$.
By sending $x$ to $\exp A$ and $y$ to $\exp B$ we define
a morphism $F \to \exp \widehat{\mathfrak{F}}$. This
morphism extends to the pro-$\k$-unipotent completion $F(\k)$
and induces an isomorphism $F(\k) \simeq \exp \widehat{\mathfrak{F}}$ :

\begin{center}
\mbox{
\xymatrix{ F \ar@{^{(}->}[dr] \ar[rr] & & \exp \widehat{\mathfrak{F}} & \raisebox{-4mm}{\hspace{-.9cm} $\subset  \k \ll A,B \gg$} \\
 &  F(\k) \ar[ur]_{\simeq} }}
\end{center}

This description makes it easier to compute with the
pro-$\ell$ completion of free groups, for the natural map $F \to F(\Q_{\ell})$
factorizes through $F^{\ell}$. Moreover, it is a classical theorem
of Lazard that the induced map $F^{\ell} \to F(\Q_{\ell})$ is injective.

\begin{center}
\mbox{
\xymatrix{ F \ar@{^{(}->}[dr] \ar[rr] & & F(\Q_{\ell})  \\
 &  F^{\ell} \ar@{^{(}->}[ur] }}
\end{center}

\section{The Grothendieck-Teichmüller groups}

\subsection{Reminder on braid groups} In order to define the Grothendieck-Teichmüller group
we need the classical presentation of $B_n$, that we recall.
In a small disk surrounding the points $i$ and $i+1$ in $D_n$ one
can define a diffeomorphism which fixes the boundary and
sends $i$ to $i+1$ by a half-turn rotation.
More precisely, a diffeomorphism of the unit disk in $\C$ exchanging
$\frac{1}{2}$ and $\frac{-1}{2}$ while fixing the
boundary is easily defined by $z \mapsto \exp^{2\ii
\pi |z|} z$, and we apply this to the chosen disk and extend it
on $D_n$ by the identity.
\begin{figure}
\begin{center}
\mbox{\resizebox{!}{5cm}{\mbox{\hspace{2cm}\includegraphics{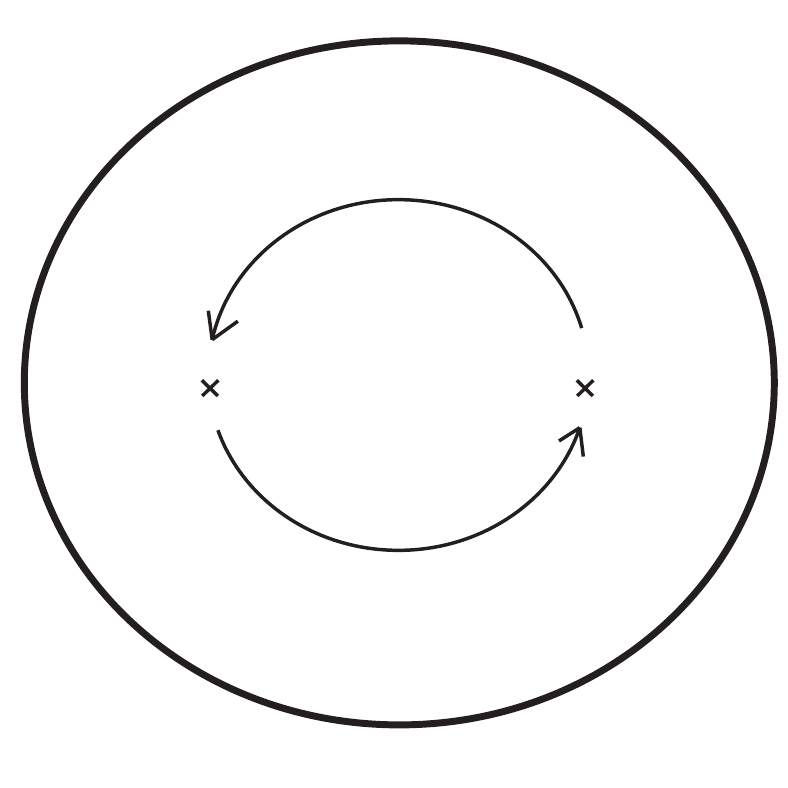}\hspace{-3cm} \includegraphics{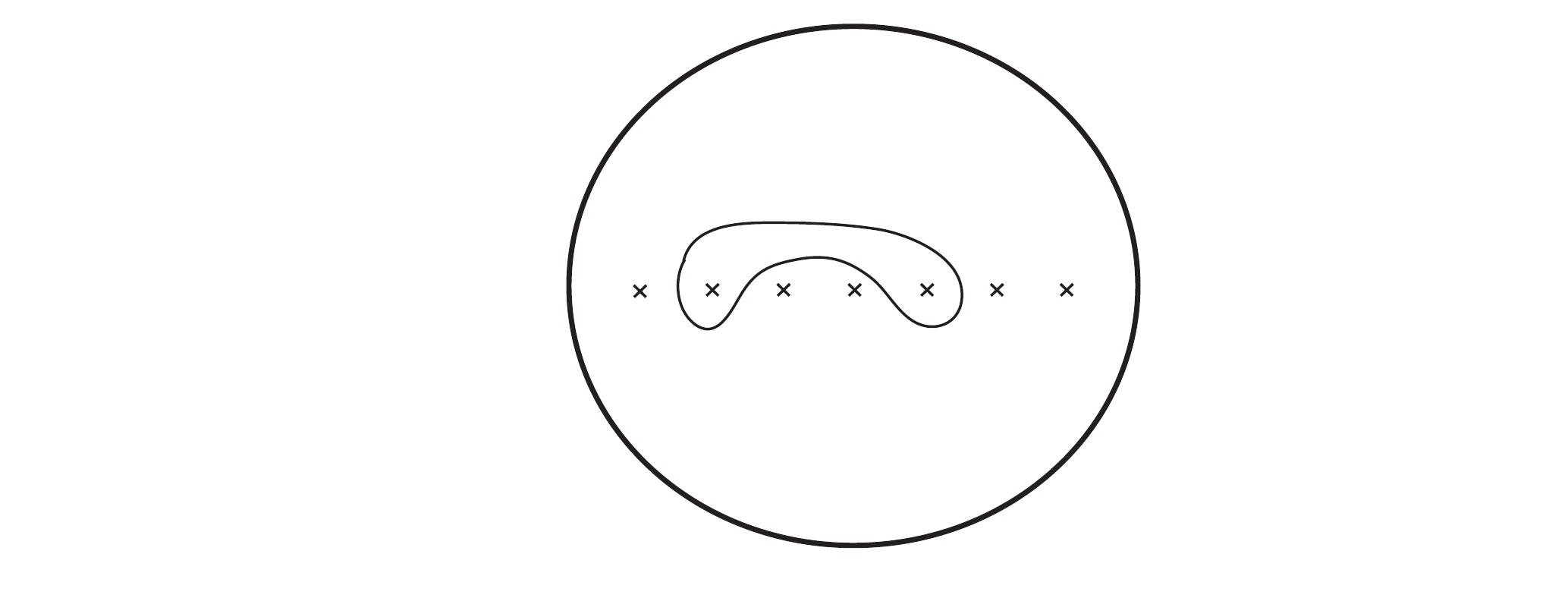}}}}
\end{center}\caption{General braid twist and $\xi_{25}$ as a Dehn twist.}
\label{xi25}
\end{figure}

We denote $\sigma_i$ the class of this diffeomorphism in $B_n$.
These elements for $1 \leq i \leq n-1$ generate $B_n$ with presentation $\sigma_i \sigma_j
= \sigma_j \sigma_i$ for $|j - i | \geq 2$ and $\sigma_i
\sigma_{i+1} \sigma_i = \sigma_{i+1} \sigma_i \sigma_{i+1}$.

The braid group $B_n$ acts on the set of punctures, thus
defining a morphism $B_n \to \mathfrak{S}_n$. The kernel of this
map is the pure braid group $P_n$. It can be seen as a mapping
class group and is generated by Dehn twists.
Recall that a Dehn twist along a simple closed curve inside $D_n$ is defined
by applying the diffeomorphism $(t,z) \to (t,e^{2\ii \pi t}z)$
of the annulus $[0,1] \times \{ z \in \C \ | \ |z| = 1 \}$
to a tubular neighbourhood of the curve and extending it trivially
outside this neighbourhood. Note that all Dehn twists of $D_n$ obviously
belong to $P_n$.

For $1 \leq i < j \leq n$ we define the usual generators
$\xi_{i,j} \in P_n$ to be the Dehn twists along a curve surrounding
$i$ and $j$ while passing above the points in-between (see figure \ref{xi25}).
They are all conjugated in $B_n$. 

\subsection{The Grothendieck-Teichmüller group}

The Grothendieck-Teichmüller group $\GTH$ is defined as a
set of couples $(\la,f) \in \ZH^{\times} \times C^1 \widehat{F}$
satisfying the following relations :
$$\begin{array}{lc}
\mathrm{(I)} & f(x,y)f(y,x)=1 \\
\mathrm{(II)} & y^{\mu}f(x,y)x^{\mu}f(z,x)z^{\mu}f(y,z)=1,
\mbox{ where } \mu = (\la-1)/2, z = (xy)^{-1} \\
\mathrm{(III)} & f(\xi_{12},\xi_{23}\xi_{24})
f(\xi_{13}\xi_{23},\xi_{34}) = f(\xi_{23},\xi_{34})f(\xi_{12}\xi_{13},\xi_{24}\xi_{34})f(\xi_{12},\xi_{23})
\mbox{ in } \widehat{P}_4 \\
\end{array}$$
Note that (II) implicitely states that $\la \in 1 + 2 \widehat{\Z}$.
Drinfeld defined in \cite{DRIN} an associative composition law
on the set of couples satisfying (I)-(II)-(III), namely
$$
(\la_1,f_1)(\la_2,f_2) = (\la_1\la_2, f_1\left(f_2 x^{\la_2}f_2^{-1},y^{\la_2}\right) f_2)
$$
for which $(1,1)$ is the neutral element. He then defined the group $\GTH$
as the subset of invertible elements for this law.

This (non-obvious)
group structure satisfies that the map $(\la,f) \to \la$
extends to an homomorphism $\GTH \to \ZH^{\times}$. One of
the main features of $\GTH$ is that it contains $\Gal(\Qb|\Q)$
and that the composite map $\Gal(\Qb | \Q) \to \ZH$ is the
cyclotomic character, whose kernel is the absolute Galois group of
the field $\Q(\mu_{\infty})$
of cyclotomic numbers. Denoting $\GTH_1$ the kernel
of the projection $\GTH \to \ZH^{\times}$ the following natural
diagram commutes.

\begin{center}
\mbox{\xymatrix{
1 \ar[r] & \GTH_1 \ar[r] & \GTH \ar[r] & \ZH^{\times} \ar[r] \ar@{=}[d] & 1 \\
1 \ar[r] & \Gal(\Qb|\Q(\mu_{\infty})) \ar@{^{(}->}[u]\ar[r] & \Gal(\Qb|\Q) \ar@{^{(}->}[u] \ar[r] & \ZH^{\times} \ar[r] & 1 
}}
\end{center}

The group $\GTH$ can be seen as an automorphism group
of the profinite completion $\widehat{B}_n$ of $B_n$,
which gives a natural description of its composition law.
For this, we need to define a collection of elements
$\delta_r$ in $B_n$, such that $\delta_2 = \sigma_1^2$,
and such that $\delta_2 \delta_3 \dots \delta_r$
is the Dehn twist associated to a curve
surrounding the first $r$ points (see figure \ref{deltadehn}).
Note that these curves can be chosen so that they do not
interesect each other. This implies that the elements
$\delta_2,\dots, \delta_r$ commute one to the other.
The collection of these curves
can be considered as a pants decomposition of $D_n$.

\begin{figure}
\begin{center}
\resizebox{!}{5cm}{\mbox{\hspace{-3.5cm} \includegraphics{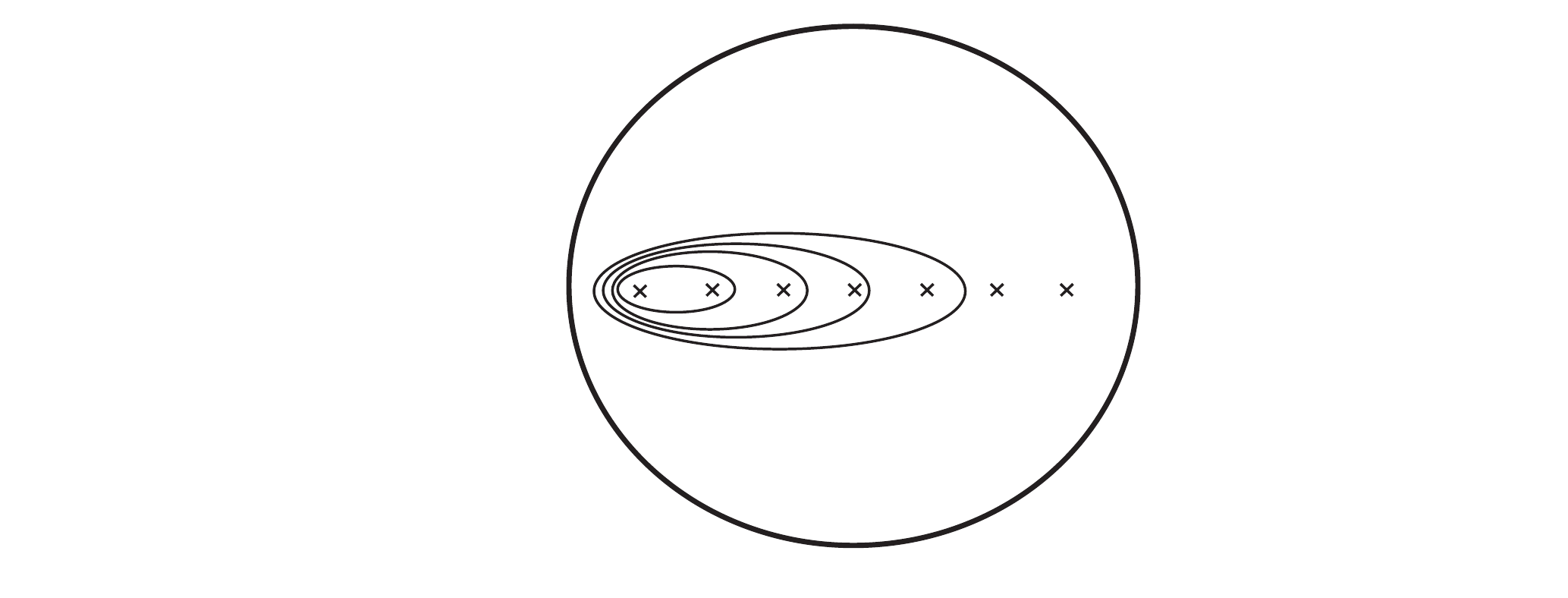}}}
\end{center}
\caption{The Dehn twists $\delta_2,\delta_2 \delta_3,\delta_2 \delta_3 \delta_4,\delta_2 \delta_3 \delta_4 \delta_5$}
\label{deltadehn}\end{figure}

The action of $(\la,f) \in \GTH$ on $\widehat{B}_n$
is defined by $\sigma_1 \mapsto \sigma_1^{\la}$ and
$$
(\la,f).\sigma_r = f(\sigma_r^2,\delta_r) \sigma_r^{\la}f(\sigma_r^2,\delta_r)^{-1}
$$
for $r \geq 2$.
This identifies $\GTH$ to a subgroup of $\widehat{B}_n$.

\subsection{Its pro-$\ell$ and pro-$\k$-unipotent versions.}

The relations defining $\GTH$ as well as the composition law make sense if $(\la,f)$
are taken in the other previously described completions
of $\Z$ and $F_2$, respectively. This in particular defines variations $\GT^{\ell}$,
and $\GT(\k)$, by considering the pro-$\ell$ and pro-$\k$-unipotent
completions of $F_2$, respectively. If we denote
$\underline{\GT}(\k)$ the semigroup of elements  $(\la,f) \in \kt \times C^1 F_2(\k)$
satisfying the analogues of (I)-(II)-(III), then $\GT(\k)$
is simply the subset of $\underline{\GT}(\k)$ defined by $\la \neq 0$.
We define $\GT_1^{\ell}$,
and $\GT_1(\k)$ in the obvious way.

We have natural morphisms
$\GTH \onto \GT^{\ell} \into \GT(\Q_{\ell})$. They are compatible with
the extensions of the cyclotomic character, namely the following
diagram commutes.
\begin{center}
\mbox{\xymatrix{
1 \ar[d] & 1\ar[d] & 1\ar[d] \\
\GTH_1\ar[d] \ar@{->>}[r] & \GT_1^{\ell}\ar[d]\ar@{^{(}->}[r] & \GT_1(\Q_{\ell})\ar[d] \\
\GTH\ar[d] \ar@{->>}[r] & \GT^{\ell}\ar[d]\ar@{^{(}->}[r] & \GT(\Q_{\ell})\ar[d] \\
\ZH^{\times}\ar[d] \ar@{->>}[r] & \Z_{\ell}^{\times}\ar[d]\ar@{^{(}->}[r] & \Q_{\ell}^{\times}\ar[d] \\
1 & 1 & 1 }}
\end{center}

Defining an action of $\GT^{\ell}$ on the pro-$\ell$
completion of $B_n$ would have little meaning, as this
latter completion is isomorphic to $\Z_{\ell}$. The same
holds for $\GT(\k)$ and the classical pro-$\k$-unipotent completion
of $B_n$. The trick used by Drinfeld in \cite{DRIN} is to define
a twisted completion $B_n(\k)$, loosely speaking, as a
semi-direct product of $\mathfrak{S}_n$ and
$$
P_n(\k) = \limpro{r} P_n/C^r P_n \otimes \k.
$$
This group $P_n(\k)$ is well-defined because the 
groups $P_n/C^r P_n$ can be proved to be torsion-free, however
it does not admit an action of $\mathfrak{S}_n$, thus one
has to use a refinement of the semi-direct product construction
instead. The Drinfeld
completion $B_n(\k)$ of $B_n$ is defined as the quotient of $B_n \ltimes
P_n(\k)$ by the subgroup generated by the elements $\sigma .\sigma^{-1}$
for $\sigma \in P_n$, considered on the left as elements of $B_n$
and on the right as elements of $P_n(\k)$.

Contrary to what happens for a genuine pro-$\k$-unipotent completion,
it is not possible to define $x^{\la}$ for an arbitrary $x \in B_n(\k)$
and $\la \in \kt$, nor $f(x,y)$ for $f \in F_2(\k)$. This however
makes sense if $x \in P_n(\k)$.
If $x^2 \in P_n(\k)$, denoting $x^{<\la>}$ the element
$x (x^2)^{\frac{\la-1}{2}}$, Drinfeld defined an action by
automorphisms
of $\GT(\k)$ on $B_n(\k)$ which sends $\sigma_1$
to $\sigma_1^{<\la>}$ and $\sigma_r$ to $f(\sigma_r^2,\delta_r) \sigma_r^{<\la>}f(\sigma_r^2,\delta_r)^{-1}
$
for $r \geq 2$. This action sends $\delta_r$ to $\delta_r^{\la}$.

\section{An example of profinite rigidity : the Kummer character $\rho_2$}

Recall that $B_3$ with generators $\sigma_1,\sigma_2$
is presented by the unique relation $\sigma_1 \sigma_2 \sigma_1
= \sigma_2 \sigma_1 \sigma_2$. We consider the representations
$R : B_3 \to \SL_2(A)$, where $A$ is a commutative integral ring
with 1, such that
$R(\sigma_1) = a$, where
$$
a = \left( \begin{array}{cc} 1 & 1 \\ 0 & 1
\end{array} \right)
$$
We have the following lemma, that we admit for now, as it will be a special
case of a more
general one.

\begin{lemma} \label{lemme41} If $R(\sigma_1)=a$, then either $R(\sigma_2)
=R(\sigma_1)$ or there
exists $u \in A$ such that $R(\sigma_2) = b(u)$, where
$$
b(u) = \left( \begin{array}{cc} 1-u  & u^2 \\ -1 & 1+u
\end{array} \right)
$$
\end{lemma}

On the other hand,
the centralizer of $a$ in $\PGL_2(A)$ is easily computed. A
complete system of representatives in $\GL_2(A)$ of this centralizer is
given by the matrices $c(v) = \left( \begin{array}{cc}
1 & v \\ 0 & 1 \end{array} \right)$ for $v \in A$.
It is easily checked that, if $b = b(0)$, then $c(v)bc(v)^{-1} = b(v)$.
Hence we have
\begin{lemma} \label{lemme42} If $R_1,R_2 : B_3 \to \SL_2(A)$ satisfy
$R_1(\sigma_1) = R_2(\sigma_1) = a$ and $R_i(\sigma_1) \neq
R_i(\sigma_2)$ for $i \in \{ 1,2 \}$, then they are conjugated
by an element of $\PGL_2(A)$.
\end{lemma}

Choose $A = \Z_{\ell}$ and define $R : B_3 \to \SL_2(\Z_{\ell})$ by
$R(\sigma_1) = a$ and $R(\sigma_2 ) = b$.
Since $\SL_2(\Z_{\ell})$ has a natural structure of a profinite group,
$R$ extends to a representation of $\widehat{B}_3$. Let
$g \in \GTH_1$. Then $R' = R \circ g$ defines a representation
of $\widehat{B_3}$ such that $R'(\sigma_1) = a$. We have
$R'(\sigma_2) \neq R'(\sigma_1)$ ; indeed, since $R = R'\circ g^{-1}$,
$R(\sigma_2)$ would otherwise be conjugated to $R(\sigma_1) = R'(\sigma_1)$
by an element of the commutator subgroup of $<R'(\sigma_1),R'(\sigma_2)>$
which is trivial, contradicting $R(\sigma_1) \neq R(\sigma_2)$.

It follows from lemma \ref{lemme42} that
there exists a well-defined $\rho(g) \in \Z_{\ell}$
such that $R'(\sigma_2 ) = c_g^{-1} b  c_g$ with $c_g = c(\rho(g))$.
This $\rho : \GTH_1 \to \Z_{\ell}$ is a continuous character
of $\GTH_1$, arising from the rigidity of the representation
$R$.

\bigskip

Since, in the profinite group $\SL_2(\Z_{\ell})$, we have,
for all $x \in \ZH$,
$$
a^x = \left(\begin{array}{cc} 1 & 1 \\ 0 & 1 \end{array} \right)^{x}
=
\left(\begin{array}{cc} 1 & x \\ 0 & 1 \end{array} \right)
$$
which is conjugated to $a$, we can extend $\rho$ to
a 1-cocycle of $\GTH$ in the same spirit. More precisely,
let $g = (\la,f) \in \GTH$. The above calculation shows that
$R \circ g(\sigma_1) = R(\sigma_1)^{\la}$.
On the other hand, lemma \ref{lemme41} can be strengthened as follows. Let
$$
a_{\la} = \left(\begin{array}{cc} 1 & \la \\ 0 & 1 \end{array} \right)
$$

\begin{lemma} Let $\la \in A$. If $R(\sigma_1)=a_{\la}$ then, either $R(\sigma_2)
=R(\sigma_1)$, or $\la \in A^{\times}$ and there
exists $u \in A$ such that $R(\sigma_2) = b_{\la}(u)$, where
$$
b_{\la}(u) = \left( \begin{array}{cc} 1-u  & \la u^2 \\ -1/\la & 1+u
\end{array} \right)
$$
\end{lemma}
\begin{proof} Since $a$ and $b$ have to be conjugate, we are looking for an
invertible matrix of trace 2, of the form
$$
b = \left( \begin{array}{cc} 1-u  & v \\ w & 1+u
\end{array} \right)
$$
such that $aba=bab$. Since the determinant of $b$ has to be $1$, we get
$vw = -u^2$.The equation $aba = bab$
yields to the system
$$
\left\lbrace \begin{array}{lcl}
u(1+\la w) & = & 0 \\
\la+\la^2 w-v+\la u^2 & = & 0 \\
-w(1+\la w) & = & 0\\
-u(1+\la w) & = & 0 \\
\end{array} \right.
$$
Since $A$ is integral, the third equation means $w = 0$ or
$\la w = -1$. But $w=0$ implies $u = 0$ and $v = 1$ hence $a = b$.
The other possibility is that $\la \in A^{\times}$ and $w=-1/m$.
It follows that the system of equations boils down to $v = \la u^2$ which is the form required.
\end{proof}

We have $c(v) b_{\la}(0) c(v)^{-1} = b_{\la}(v\la^{-1})$. Introduce the
right action of $A^{\times}$ on $\SL_2(A)$ induced by the morphism
$A^{\times} = \GL_1(A) \to \GL_2(A)$ which sends $\la$
to 
$$d_{\la} = \left( \begin{array}{cc} 1 & 0 \\ 0 & \la \end{array} \right)
$$
and by the action by conjugation of $\GL_2(A)$ on $\SL_2(A)$,
that is $x.\la = d_{\la}^{-1} x d_{\la}$. Then
$b_{\la}(u) = d_{\la}^{-1} b(u) d_{\la} = b(u).\la$.
It follows that $R \circ g(\sigma)
= c_g^{-1} (R(\sigma).\la) c_g$ for all $\sigma \in \widehat{B}_3$
and some well-defined $c_g = c_{\rho(g)}$ with $\rho(g) \in A$, and it is
easily checked that $c_{g_1g_2} = (c_{g_1}.\la_2) c_{g_2}$ hence
$\rho(g_1 g_2) = \la_2 \rho(g_1)+\rho(g_2)$.
This shows that $\rho$ can be extended to a continuous 1-cocycle
of $\GTH$ with values in $\Z_{\ell}$, endowed with the action by
multiplication induced by $\GTH \to \Z_{\ell}^{\times}$.

\bigskip

For $\sigma \in \Gal(\Qb|\Q)$, $\rho$ coincides with
the $\ell$-part of $-8 \rho_2$, where $\rho_2 : \Gal(\Qb|\Q) \to \ZH$ is the
Kummer character of (positive) roots of 2, defined by
$$
\sigma(\sqrt[n]{2}) = \zeta_n^{\rho_2(\sigma)} \sqrt[n]{2}
$$
where $\zeta_n = \exp(2 \ii \pi/n)$. This can be inferred from
\cite{NAKA}, corollary 4.13, where $f(a^2,b^2)$
is calculated in $\SL_2(\ZH)$ for $(\la,f)$ the
image in $\GTH$ of $\sigma \in \Gal(\Qb|\Q)$.
Using this calculation we find in this situation
$$
R\circ g(\sigma_2) = f(a^2,b^2)^{-1} b^{\la} 
f(a^2,b^2) = b_{\la}(8 \rho_2 \la^{-1}) = c(-8\rho_2)^{-1}
(b.\la) c(-8\rho_2)
$$
where $\rho_2 = \rho_2(\sigma)$.

It has been shown in \cite{NAKASCHN}
that $\rho_2$ can be extended to $\widehat{GT}$.
This construction thus gives an alternative simple way to do so.

\medskip

Note that the morphism $B_3 \to \GL_2(\Z_{\ell})$ used here can be seen
as a degeneration for $q =-1$ of the classical Burau representation
$$
\sigma_1 \mapsto \left( \begin{array}{cc} -q & 1 \\
0 & 1 \end{array} \right)
\ \ \ 
\sigma_2 \mapsto \left( \begin{array}{cc} 1 & 0 \\
q & -q \end{array} \right)
$$

\section{GT-rigidity}

Let $\k$ be a field of characteristic 0, $\mathcal{A} = \k[[h]]$
the ring of formal power series and
$K = \k((h))$ the field of formal Laurent series. We let $\GL_N^0(K) = \{ x \in \GL_N(K) \ | \ x \equiv 1
\mod h \} \subset \GL_N(\mathcal{A}) $.

We consider representations $R : B_n \to \GL_N(K)$
such that $R(P_n) \subset \GL_N^0(K)$. This implies that $R(B_n)
\subset \GL_N(\mathcal{A})$, as is readily checked. Such a representation
can be naturally extended to $B_n(\k)$.

A natural idea would be to call an irreducible representation $R : B_n \to \GL(V)$
GT-rigid if $R \circ g \simeq R$ for all $g \in \GT(\k)$. However,
this notion has little interest, as shows the following proposition.

\begin{prop} If $R \circ g \simeq R$ for all $g \in \GT(\k)$, then
$R(P_n) = \{ 1 \}$ and $R \circ g = R$ for all $g \in \GT(\k)$.
\end{prop}
\begin{proof}
We know that the morphism $\GT(\k) \to \k^{\times}$ is surjective. Let
$\la \mapsto g_{\la}$ be a set-theoretical section of this
morphism. By assumption we have
$$
R(\sigma_1^2)^{\la} 
= R((\sigma_1^2)^{\la}) = R(g_{\la} . \sigma_1^2 ) \sim R(\sigma_1^2)
$$
for all $\la \in \k^{\times}$.
By comparing spectrums in some algebraic closure of $K$ it follows that,
since $\k$ is infinite, $R(\sigma_1^2)$ is unipotent.
Thanks to a result of Drinfeld (see \cite{DRIN} propositions 5.2 and 5.3), we can choose $\la \mapsto
g_{\la}$ to be an algebraic section of $\GT(\k) \to \k^{\times}$.
Writing $R(\sigma_1^2)^{\la} 
= (1+N)^{\la}$ with $N$ nilpotent we show that $R(\sigma_1^2)^{\la}$ tends
to 1 when $\la \to 0$ namely that the 1-parameter subgroup
$\la \mapsto R \circ g_{\la}(\sigma_1^2)$ can be algebraically extended
to a morphism from $\k$ which sends $0$ to $1 \in \GL_N(K)$.

More generally, the map $\la \mapsto R \circ g_{\la} \in \Hom(B_n(\k),
\GL_N(K))$ can be chosen to be a polynomial mapping that can be specialized in
$\la = 0$. Indeed, the map $\la \mapsto g_{\la}$ as constructed
in \cite{DRIN} extends to a polynomial mapping $\k \to \underline{GT}(\k)$
(see \cite{DRIN} \S 5 remark 2 after proposition 5.9) and $\underline{GT}(\k)$
acts by endomorphisms on $B_n(\k)$. It follows that the
image of any generator $\xi_{i,j}$ of $P_n$ under $R \circ g_{\la}$
can be written as $1 + N_{i,j}(\la)$ where $N_{i,j}(\la)$ is a nilpotent
matrix with values in $K[\la]$.

Let us consider the character $\chi_{\la}$ of the restriction
of $R \circ g_{\la}$ to $P_n$. For all $\sigma \in P_n$
we know that $R\circ g_{\la}$ has coefficients in $K[\la]$
hence $\chi_{\la}(g) \in K[\la]$. By assumption it is independent
of $\la \in \kt$ and equals $\dim V$ for $\la = 0$, whence we get
that it is the character of the trivial representation. Since $K$
has characteristic 0, characters determine isomorphism
classes of the semisimplication of representations. It follows that
the restriction of $R$ to $P_n$ is a unipotent representation.
On the other hand, it is the restriction to a normal subgroup
of a (semi)simple representation of $B_n$, hence it is semisimple.
It follows that $R(P_n) = \{ 1 \}$. By the formulas defining
the action of $\GT(\k)$ it also follows that $R \circ g = R$
for all $g \in \GT(\k)$.

\end{proof}

This proof shows that an obstacle comes from the
quotient map $\pi : GT(\k) \to \k^{\times}$. It is the only one,
namely there exist representations $R$ such that $R \circ g
\simeq R$ for all $g$ in the kernel $\GT_1(\k)$
of this map.

\begin{defi} $R$ is called $\GT_1$-rigid if $\ \forall g \in \GT_1(\k)
\ \ \ R \circ g \simeq R$.
\end{defi}

If $R$ is such a representation we thus get a
projective representation $Q_R : \GT_1(\k) \to \PGL(V)$. Consider
now the right action of $\k^{\times}$ on $K$ defined by
$(f .\alpha)(h) = f(\alpha h)$ for $f(h) \in K$ and $\alpha \in 
\k^{\times}$. It induces an action of $\GT(\k)$ on $K$ hence
on $\PGL(V)$ by the
morphism $\pi : \GT(\k) \to \kt$.

\begin{defi} $R$ is called $\GT$-rigid if $\ \forall g \in \GT(\k)
\ \ \ R \circ g \simeq ( \sigma \mapsto R(\sigma) .\pi(g))$.
\end{defi}

In particular, if $R$ is $\GT$-rigid then it is $\GT_1$-rigid. It turns
out that such representations exist. They provide
(non-commutative) cocycles in $Z^1(\GT(\k), \PGL_N(K))$.

\section{Burau representation}

\subsection{The Burau characters}

Let $q = e^h \in K$. We define the classical (reduced) Burau
representation by the following formulas. Let $e_1,\dots,
e_{n-1}$ be the canonical basis of $K^{n-1}$.
We let $\sigma_1 . e_1 = q e_1$ and $\sigma_1.e_m
= -q^{-1} e_m$ for all $m \neq 1$. For $2 \leq d \leq r-1$
we let $\sigma_d .e_m = -q^{-1} e_m$ if
$m \not\in \{ d, d+1 \}$, and we make $\sigma_d$
act on the plane with basis $(e_{d-1},e_d)$
by the matrix
$$
\frac{1}{[d]_q} \left( \begin{array}{cc} -q^{-d} & * \\
* & q^d \end{array} \right)
$$
with
$$
* = \sqrt{ [d+1]_q [d-1]_q } \mbox{\ \  and \ \   } [\alpha]_q = \frac{q^{\alpha}
- q^{-\alpha}}{q-q^{-1}}.
$$
This defines an irreducible $(n-1)$-dimensional
representation of $B_n$, called the Burau representation.
It can be shown (see \cite{GT}) that this representation is $\GT$-rigid, hence we get a
projective representation $Q_R : \GT_1(\k) \to \PGL_{n-1}(K)$.

The elements $\delta_r$ act in the following way :
$$
\left\lbrace \begin{array}{lcl}
\delta_r e_{r-1} & = & q^2 e_{r-1} \\
\delta_r e_s & = & q^{-2(r-2)} e_s \mbox{ if } s < r-1 \\
\delta_r e_s & = & q^{-2(r-1)} e_s \mbox{ if } s > r-1.
\end{array} \right.
$$
It is easily checked that the algebraic hull (Zariski closure)
of the subgroup
generated by the corresponding matrices is the set of diagonal matrices,
that is a maximal torus of $\GL_{n-1}(K)$. Since the $\delta_r$'s
are fixed by $\GT_1(\k)$, it follows that
the image of $Q_R$ centralizes this torus, hence is part of it.
We thus can write $Q_R(g)$ as the image in $\PGL_{n-1}(K)$
of a well-defined matrix of the form
$$
\left( \begin{array}{cccc}
1 & 0 & 0 \\
0 & \chi_2(g) & 0\\
0 & 0 & \chi_2 \chi_3(g) \\ 
 & & & \ddots 
\end{array} \right)
$$
where $\chi_d : \GT_1(\k) \to K^{\times}$ is a character
that can be extended to a 1-cocycle $\GT(\k) \to K^{\times}$.
This decomposition of $Q_R$ is compatible with the usual
embedding $B_n \subset B_{n+1}$,
so we do not need to specify $n$ and, as $n$ varies, this defines
an infinite family $\chi_d, d \geq 2$ of characters of
$\GT_1(\k)$.

\subsection{Transcendental aspects}

Recall from \cite{DRIN} that $\Phi_{KZ}$ is defined as $G_1^{-1} G_0$,
where $G_1, G_0$ are the solutions over $]0,1[$ of the differential
equation
$$
G'(x) = \frac{1}{2 \ii \pi} \left( \frac{A}{x} + \frac{B}{x-1} \right)
G(x)
$$
with values in $\C \ll A,B \gg$, defined by the asymptotic conditions $G_0(x) \sim x^{\frac{A}{2 \ii \pi}}$
when $x \to 0$ and $G_1(x) \sim (1-x)^{\frac{B}{2 \ii \pi}}$
when $x \to 1$. Moreover, $\Phi_{KZ}$ as well as $\overline{\Phi_{KZ}}(A,B)
= \Phi_{KZ}(-A,-B)$ belong to the $\C$-points of an affine scheme
$\Ass_1$ over $\Q$ which is a torsor under $\GT_1$. In particular
there exists a well-defined $g_0 \in \GT_1(\C)$ such that
$\overline{\Phi_{KZ}} = g_0. \Phi_{KZ}$. Using the work of
J. Gonz\'alez-Lorca in \cite{JORGE}, it was shown in \cite{GT} that
$$
\chi_d(g_0) = \exp\left(2 \sum_{n=1}^{+\infty} \frac{\zeta(2n+1)}{2n+1}
\hslash Q_n(d) \right)
$$
where $Q_n(d) = (d+1)^{2n+1} + (d-1)^{2n+1} - 2 d^{2n+1}$ and $\hslash = h/\ii \pi$.

\subsection{Arithmetic aspects}

On the
arithmetic side, if we fix $\ell$ and denote $\overline{\sigma}$
the image of $\sigma \in \Gal(\overline{\Q}| \Q(\mu_{\ell^{\infty}}))$
in $\GT_1(\Q_{\ell})$, we have
$$
\chi_d(\overline{\sigma}) = 1 - 8 \kappa_3^*(\sigma) d h^3 - \frac{8}{3} \kappa_5^*
(\sigma) d(1+2d^2) h^5 + \dots
$$
where $\kappa_m^*(\sigma) = \kappa_m(\sigma)/({\ell}^{m-1} -1)$
and $\kappa_m$ are $\Z_{\ell}$-valued characters called the Soulé
characters. They can be elementary defined as follows (see \cite{ICHI}).
Let $\zeta_n$ denote a primitive $\ell^n$-th root of 1, and
introduce
$$
\eps_{m,n} = \prod_{\stackrel{0 < a < \ell^n}{a \wedge {\ell} = 1}}
(\zeta_n^a -1)^{[a^{m-1}]}
$$
where $[a^{m-1}]$ denotes the (non-negative) euclidean remainder of the division of
$a^{m-1}$ by $\ell^n$. It can be proved that these
elements are totally real and totally positive, and $\kappa_m$
is defined by
$$
\sigma(\sqrt[\ell^n]{\eps_{m,n}}) =
\zeta_n^{\kappa_m(\sigma)} \sqrt[\ell^n]{\sigma(\eps_{m,n})} \zeta_n^{\kappa_m(\sigma)}.
$$
Taking the logarithm of $\chi_d$, this proves that the Soulé characters,
like the Kummer character $\rho_2$, can also be extended to $\GTH$, at least for $m \in \{ 3, 5 \}$.

\end{document}